\documentclass[11pt]{amsart}

\usepackage[T1]{fontenc}
\usepackage[latin1]{inputenc}
\usepackage[english]{babel}
\usepackage{amsmath,amssymb,amsthm}
\usepackage{enumitem}
\renewcommand{\MR}[1]{}
\usepackage{mymacros}
\usepackage[margin=1.5in]{geometry}
\usepackage{hyperref}

\author{Rapha\"el Clou\^atre}
\address{Department of Mathematics, University of Manitoba, Winnipeg, Manitoba, Canada R3T 2N2}
\thanks{R.C. was partially supported by an NSERC Discovery grant.}
\email{raphael.clouatre@umanitoba.ca}
\author{Michael Hartz}
\address{Fakult\"at f\"ur Mathematik und Informatik, FernUniversit\"at in Hagen, 58084 Hagen, Germany}
\email{michael.hartz@fernuni-hagen.de}
\author{Dominik Schillo}
\address{Fachrichtung Mathematik, Universit\"at des Saarlandes, Postfach 151150,
66041 Saarb\"ucken, Germany}
\email{schillo@math.uni-sb.de}
\title[A Beurling theorem for spaces with a CNP factor]{A Beurling--Lax--Halmos theorem for spaces with a complete Nevanlinna--Pick factor}


\begin{document}

\begin{abstract}
  We provide a short argument to establish a Beurling--Lax--Halmos theorem for reproducing
  kernel Hilbert spaces whose kernel has a complete Nevanlinna--Pick factor.
  We also record factorization results for pairs of nested invariant subspaces.
\end{abstract}

\maketitle

\section{Introduction}

The Beurling--Lax--Halmos theorem is a fundamental result connecting
operator theory with function
theory on the unit disc.
It shows that a non-zero closed subspace $\cM$
of the vector-valued Hardy space $H^2 \otimes \cE$ is invariant under multiplication
by the coordinate function $z$ if and only if there exists a Hilbert space $\cF$ and an
isometric multiplier $\Phi$ from $H^2 \otimes \cF$ to $H^2 \otimes \cE$ such that
$\cM = \Phi \cdot (H^2 \otimes \cF)$; see \cite{Helson64} for a classical treatment.
This theorem was extended to complete Nevanlinna--Pick spaces by McCullough and Trent \cite{MT00};
in this setting, the multiplier $\Phi$ is in general only partially isometric, and moreover the dimension of $\cF$ may exceed the dimension of $\cE$. Similarly, Ball and Bolotnikov \cite{Ball2013,BB13a} considered invariant subspaces of weighted Bergman spaces $A_n$
on the unit disc and showed that they can be represented as ranges of partially isometric
multipliers from $H^2 \otimes \cF $ into $A_n \otimes \cE$.

In a different direction,
an intrinsic representation of invariant subspaces of the classical Bergman space in terms
of wandering subspaces had previously been established in deep work by
Aleman, Richter and Sundberg \cite{Aleman1996}. This representation
was generalized to other spaces related to the classical Bergman space
by Shimorin \cite{Shimorin01,Shimorin03} and by McCullough and Richter \cite{MR02}.

Returning to Beurling-type theorems involving vector-valued partially isometric
multipliers, the result of Ball and Bolotnikov has been extended by several authors.
In \cite{Sarkar2015}, the Bergman space is replaced with a more general reproducing kernel Hilbert
space of holomorphic functions for which $z$ is a contractive multiplier.
The paper \cite{Sarkar2016} further extends this result to the unit ball in $\bC^d$, where the role of $H^2$ is played by the Drury--Arveson space. This last generalization, along
with a uniqueness statement, is also obtained in \cite{Bhattacharjee2017} as part
of a wider investigation of dilations and wandering subspaces.

The purpose of this note is to exhibit a general Beurling--Lax--Halmos theorem for reproducing kernel Hilbert spaces whose kernel has a complete Nevanlinna--Pick factor.
The proof consists of general arguments involving reproducing kernels.
Using this approach, kernels with a complete Nevanlinna--Pick factor are no more difficult to treat
than complete Nevalinna--Pick kernels themselves.

The prototypical example of a reproducing kernel Hilbert space with a normalized complete Nevanlinna--Pick
kernel is the Hardy space $H^2$. Other examples include the Drury--Arveson space $H^2_d$ on the
unit ball in $\bC^d$ \cite{AM00,Arveson98} and (weighted) Dirichlet spaces \cite{Shimorin02}.
The reproducing kernel of the Bergman space on the unit disc is not a complete Nevanlinna--Pick kernel,
but it is the square of the kernel of $H^2$, so the kernel of $H^2$ is a complete Nevanlinna--Pick
factor.
Background information about reproducing kernel Hilbert spaces can be found in \cite{PR16},
for complete Nevanlinna--Pick spaces, see \cite{AM02}.

We now introduce some terminology and notation.
Let $X$ be a set and let \[k: X \times X \to \bC\] be a function.
If $k$ is positive semi-definite, in the sense that the $n\times n$ complex matrix $[k(x_i,x_j)]$ is positive semi-definite for every finite subset $\{x_1,\ldots,x_n\}\subset X$, then we write $k \ge 0$ and say that $k$ is a \emph{kernel}. The corresponding \emph{reproducing kernel Hilbert space} on $X$
is denoted by $\cH_k$. Thus,
\begin{equation*}
  \langle f, k(\cdot,x) \rangle = f(x) \quad \text{ for all } f \in \cH \text{ and } x \in X.
\end{equation*}
We write $\Mult(\cH_k)$ for the multiplier algebra of $\cH_k$, consisting of those functions $\varphi:X\to \bC$ such that $\varphi \cH_k\subset \cH_k$.
A kernel $k$ is said to be \emph{normalized} if there exists a point $x_0 \in X$ with $k(x,x_0) = 1$
for all $x \in X$. The main result of \cite{AM00} shows that a normalized kernel $s$ is a \emph{complete Nevanlinna--Pick}
kernel if and only if $s$ is non-vanishing and $1 - 1 /s \ge 0$.

We will also require vector-valued versions of the aforementioned objects. Given a Hilbert space $\cE$,
we may regard elements of $\cH_k \otimes \cE$ as $\cE$-valued
funtions on $X$. If $\cF$ is another Hilbert space and $\ell$ is another kernel on $X$, we write $\Mult(\cH_\ell \otimes \cF, \cH_k \otimes \cE)$
for the space of all $\cB(\cF,\cE)$-valued functions on $X$ that multiply $\cH_\ell \otimes \cF$ into $\cH_k \otimes \cE$.
A multiplier $\Phi \in \Mult(\cH_\ell \otimes \cF, \cH_k \otimes \cE)$ is said
to be contractive (respectively partially isometric)
if the associated multiplication operator $M_{\Phi}: \cH_\ell \otimes \cF \to \cH_k \otimes \cE$ is contractive (respectively partially isometric).

Next, we describe our main results, which will all be proven in Section \ref{sec:proofs}. Our general Beurling--Lax--Halmos theorem applies to pairs of kernels $k$ and $s$, where $s$ is a complete
Nevanlinna--Pick kernel and $k / s \ge 0$.
It follows from the Schur product theorem and a standard characterization of multipliers (see Lemma \ref{lem:mult_char} below) that if $k/s \ge 0$, then
$\Mult(\cH_s) \subset \Mult(\cH_k)$, and that the inclusion is contractive.
We say that a subspace $\cM \subset \cH_k \otimes \cE$ is $\Mult(\cH_s)$-invariant if $(M_\varphi \otimes I) \cM \subset \cM$ for all $\varphi \in \Mult(\cH_s)$.

\begin{thm}
  \label{thm:main}
  Let $X$ be a set, let $k$ be a kernel on $X$ and let $s$ be a normalized complete Nevanlinna--Pick kernel on $X$
  such that $k/s \ge 0$.
  Let $\cE$ be a Hilbert space and let
  $\cM \subset \cH_k \otimes \cE$ be a non-zero closed subspace.
  The following are equivalent:
  \begin{enumerate}[label=\normalfont{(\roman*)}]
    \item The subspace $\cM$ is $\Mult(\cH_s)$-invariant.
    \item There exist an auxiliary
  Hilbert space $\cF$ and a partially isometric multiplier $\Phi \in \Mult(\cH_s \otimes \cF, \cH_k \otimes \cE)$
  such that
  \begin{equation*}
    \cM = \Phi \cdot (\cH_s \otimes \cF).
  \end{equation*}
  \end{enumerate}
\end{thm}

Specializing to the case where $k=s$ recovers the result of McCullough and Trent \cite{MT00}.
As explained in \cite[Theorem 0.7]{MT00}, the Nevanlinna--Pick assumption on $s$ is already necessary in this special case. It is natural to wonder whether the representation obtained in Theorem \ref{thm:main} is unique in some appropriate sense. We show that this is indeed the case in Proposition \ref{prop:unique}, thereby generalizing \cite[Theorem 4.2]{MT00} and \cite[Theorem 6.5]{Bhattacharjee2017}.

\begin{rem}
  \label{rem:row_contr}
If $\cH_s = H^2$ is the Hardy space on the unit disc, then $k/s \ge 0$ if and only if the identity function
  $z$ is a contractive multiplier of $\cH_k$. More generally,
  if $\cH_s = H^2_d$ is the Drury--Arveson space, then $k/s \ge 0$
  if and only if the coordinate functions $z_1,\ldots,z_d$ form a row
  contraction on $\cH_k$, see Lemma \ref{lem:mult_char} (a) below. More examples of pairs of spaces that satisfy this positivity condition can
be found in \cite[Subsection 2.2]{AHM+17c}.
\end{rem}

More broadly, the positivity condition $k/s \ge 0$, where $s$ is a normalized complete Nevanlinna--Pick kernel,
has a number of function theoretic and dilation theoretic implications (see, for
instance,  \cite{AHM+17c} and \cite{CH18}). 

In concrete cases of spaces of holomorphic functions on the unit ball in $\bC^d$,
the assumption of $\Mult(\cH_s)$-invariance in Theorem \ref{thm:main} can often be weakened to invariance under multiplication by the coordinate functions, as we show in Proposition \ref{prop:coord}. This applies in particular to the Drury--Arveson space or more generally to a
unitarily invariant space on $\bB_d$ (see \cite[Section 4]{GRS2002} or \cite[Section 7]{hartz2017isom}).
In light of Remark \ref{rem:row_contr}, we therefore
recover the Beurling--Lax--Halmos theorems of \cite{Ball2013,Bhattacharjee2017,Sarkar2015,Sarkar2016} in this case.

Given a pair of nested invariant subspaces, one can show that there exists
a factorization of the corresponding representations in the form of the following result. The case $k=s$ is
also due to McCullough and Trent \cite[Theorem 0.14]{MT00}.

\begin{thm}
  \label{thm:factorization}
  Assume the setting of Theorem {\rm \ref{thm:main}}. Let $\cN \subset \cM \subset \cH_k \otimes \cE$
  be two non-zero closed subspaces and let $\cF, \cG$ be Hilbert spaces. If $\Phi \in \Mult(\cH_s \otimes \cF,
  \cH_k \otimes \cE)$ and $\Psi \in \Mult(\cH_s \otimes \cG, \cH_k \otimes \cE)$ are partially
  isometric multipliers with $\cM = \Phi \cdot (\cH_s \otimes \cF)$ and $\cN = \Psi \cdot (\cH_s \otimes \cG)$,
  then there exists a contractive multiplier $\Gamma \in \Mult(\cH_s \otimes \cG, \cH_s \otimes \cF)$
  with $\Psi = \Phi \Gamma$.
\end{thm}

In fact, in Theorem \ref{thm:factorization_three_space} we will establish a more general version of the above result, involving three kernels $k,\ell$ and $s$.

McCullough and Trent constructed an example to show that even when $k=s$, the contractive multiplier $\Gamma$ in Theorem \ref{thm:factorization}
cannot, in general, be taken to be partially isometric if $\Psi$ is specified in advance, see Section 5 of \cite{MT00}. They conjectured, however, that there always
exists some choice of $\Psi$ such that $\Gamma$ can be taken to be partially isometric, see \cite[Conjecture 5.1]{MT00}
for the precise statement.
This conjecture was settled in the affirmative by Arias \cite[Conjecture 6.1]{arias2004}.
We extend Arias's result to pairs of spaces, and in doing so give an alternative proof of his result
in the language of reproducing kernel Hilbert spaces.

\begin{thm}
  \label{thm:factorization2}
  Assume the setting of Theorem {\rm \ref{thm:main}}. Let $\cN \subset \cM \subset \cH_k \otimes \cE$
  be two non-zero closed subspaces and let $\cF$ be a Hilbert space. If $\Phi \in \Mult(\cH_s \otimes \cF,
  \cH_k \otimes \cE)$ is a partially isometric multiplier with $\cM = \Phi \cdot (\cH_s \otimes \cF)$,
  and if $\cN$ is $\Mult(\cH_s)$-invariant, then there exist a Hilbert space $\cG$
  and a partially isometric multiplier $\Gamma \in \Mult(\cH_s \otimes \cG, \cH_s \otimes \cF)$
  such that $\Phi \Gamma \in \Mult(\cH_s \otimes \cG, \cH_k \otimes \cE)$ is a partially isometric multiplier and $\cN = (\Phi \Gamma) \cdot (\cH_s \otimes \cG)$.
\end{thm}

In the classical theory of multiplier invariant subspaces of $H^2$ and of inner functions, boundary values play an important role.
Greene, Richter and Sundberg \cite{GRS2002} studied boundary values of partially isometric multipliers of
many complete Nevanlinna--Pick spaces on the unit ball in $\bC^d$. They were thus able to strenghten
the analogy between the theorem
of McCullough and Trent and the classical Beurling--Lax--Halmos theorem. In Section \ref{sec:further}, we discuss
the possibility of boundary value results for pairs of spaces.

\section{Proofs and additional results}
\label{sec:proofs}

\subsection{Preliminary lemmas}

We require the following well-known lemma. 
Given a Hilbert space $\cE$, a \emph{$\cB(\cE)$-valued kernel} on $X$
is a positive semi-definite function $K: X \times X \to \cB(\cE)$. As in the scalar
case where $\cE=\bC$, we denote the associated reproducing kernel Hilbert space of $\cE$-valued
functions by $\cH_K$.
Note that if $k:X\times X\to \bC$ is a usual scalar-valued kernel, then $\cH_k\otimes \cE$ is the reproducing kernel Hilbert space associated to the $\cB(\cE)$-valued kernel $k I_{\cE}$.

\begin{lem}
  \label{lem:mult_char}
  Let $\cE,\cF$ be Hilbert spaces, let $E$ be a $\cB(\cE)$-valued kernel on $X$ and let $F$ be a $\cB(\cF)$-valued
  kernel on $X$. Let $\Phi: X \to \cB(\cE,\cF)$ be a function and define
  \begin{equation*}
    L(z,w) = F(z,w) - \Phi(z) E(z,w) \Phi(w)^* \quad (z,w \in X).
  \end{equation*}
  Then,
  \begin{enumerate}[label=\normalfont{(\alph*)}]
    \item $\Phi$ is a contractive multiplier from $\cH_E$ to $\cH_F$ if and only if $L \ge 0$.
    \item $\Phi$ is a co-isometric multiplier from $\cH_E$ to $\cH_F$ if and only if $L = 0$.
  \end{enumerate}
\end{lem}

\begin{proof}
  Part (a) is essentially contained in \cite[Theorem 6.28]{PR16}.
  To prove (b), we may assume by (a) that $\Phi$ is a multiplier. Let $\xi,\eta \in \cF$,
  and let $E_w = E(\cdot,w)$ and $F_w = F(\cdot,w)$.
  Then
  \begin{align*}
    \langle  \Phi(z) E(z,w) \Phi(w)^* \xi,\eta \rangle_{\cF} &=
    \langle E_w \Phi(w)^* \xi, E_z \Phi(z)^* \eta \rangle_{\cH_E}\\
    &= \langle M_\Phi^* F_w \xi, M_\Phi^* F_z \eta \rangle _{\cH_E},
  \end{align*}
  whereas 
  \[
    \langle F(z,w) \xi, \eta \rangle_{\cF} = \langle F_w \xi, F_z \eta \rangle_{\cH_F}
  \]
   from which the result follows.
\end{proof}

Observe that if $\cM\subset \cH_k\otimes \cE$ is a closed subspace, then it is itself a reproducing kernel Hilbert space.
The scalar-valued version of the following lemma can be found in \cite[Proposition 2.4]{AHM+17c}.
The proof of the general case is almost identical. For the convenience of the reader, we provide the short argument.

\begin{lem}
  \label{lem:invariant_subspace}
  In the setting of Theorem {\rm \ref{thm:main}}, let $k^{\cM}$ denote the $\cB(\cE)$-valued reproducing
  kernel of $\cM$. Then $k^{\cM} / s \ge 0$.
\end{lem}

\begin{proof}
  Since $s$ is a complete Nevanlinna--Pick kernel, there exists by the main result of \cite{AM00}
  a Hilbert space $\cL$
  and a function $b: X \to \cB(\cL,\bC)$ such that 
  \[
  s(z,w) = (1 - b(z) b(w)^*)^{-1}.
  \]
  Furthermore, the assumption $k/s \ge 0$ implies that
  $b$ is a contractive multiplier from $\cH_k \otimes \cL$ to $\cH_k$ by part (a) of Lemma \ref{lem:mult_char}.
  Thus, if we define $B (z) = \id_{\cE} \otimes b(z)$ for $z \in X$, then it is readily verified that $B$ is a contractive multiplier from $\cH_k \otimes \cE \otimes \cL$
  to $\cH_k \otimes \cE$.
  Now, $\Mult(\cH_s)$-invariance of $\cM$ implies that $B$ multiplies $\cM \otimes \cL$ into $\cM$.
  Indeed, $b$ is a contractive multiplier from $\cH_s \otimes \cL$ to $\cH_s$ (by part (a) of Lemma \ref{lem:mult_char}), so for every vector $\eta \in \cL$, the function
  $z\mapsto b(z)\eta$
  is a multiplier of $\cH_s$. Consequently, if $F \in \cM$ then
  \[
  B \cdot (F \otimes \eta) = b(\cdot) \eta F \in \cM.
  \]
  Thus, $B$ is also
  a contractive multiplier from $\cM \otimes \cL$ to $\cM$, and yet another application of part (a) of Lemma \ref{lem:mult_char} reveals that $k^{\cM} / s \ge 0$ since
  \[
  (k^{\cM}/s)(z,w)=k^{\cM}(z,w)-B(z) (k^{\cM}(z,w) \otimes I_{\cL})B(w)^*. \qedhere
  \]
\end{proof}

We emphasize here that the previous lemma is where we crucially make use of the fact that $s$ is a complete Nevanlinna-Pick kernel.

\subsection{Beurling--Lax--Halmos theorem and uniqueness}
Our Beurling--Lax--Halmos theorem is an immediate consequence of the preceding two lemmas.
\begin{proof}[Proof of Theorem \ref{thm:main}]
  We prove the non-trivial direction (i) $\Rightarrow$ (ii).
  Let $k^{\cM}$ denote the $\cB(\cE)$-valued reproducing kernel of $\cM$. Lemma \ref{lem:invariant_subspace}
  implies that $k^{\cM} / s \ge 0$, hence there exist a Hilbert space $\cF$
  and a function $\Phi: X \to \cB(\cF,\cE)$ such that
  \begin{equation}
    \label{eqn:quotient_factored}
    \frac{k^{\cM}(z,w)}{s(z,w)} = \Phi(z) \Phi(w)^* \quad (z,w \in X).
  \end{equation}
  Part (b) of Lemma \ref{lem:mult_char} shows that $\Phi$ is a co-isometric
  multiplier from $\cH_s \otimes \cF$ to $\cM$, which finishes the proof.
\end{proof}

We now formulate and prove the uniqueness statement.

\begin{prop}
  \label{prop:unique}
  Assume the setting of Theorem {\rm \ref{thm:main}}.
  \begin{enumerate}[label=\normalfont{(\alph*)}]
\item If
  $\Phi \in \Mult(\cH_s \otimes \cF, \cH_k \otimes \cE)$ and
  $\widetilde \Phi \in \Mult(\cH_s \otimes \widetilde \cF, \cH_k \otimes \cE)$ are two
  partially isometric multipliers as in Theorem {\rm \ref{thm:main}}, then there exists
  a partial isometry $V: \cF \to \widetilde \cF$ with
  \begin{equation*}
    \Phi(z) = \widetilde \Phi(z) V \quad \text{ and } \quad \widetilde \Phi(z) = \Phi(z) V^* \quad \text{ for all } z \in X.
  \end{equation*}
\item The multiplier $\Phi \in \Mult(\cH_s \otimes \cF, \cH_k \otimes \cE)$ in Theorem {\rm \ref{thm:main}} can be chosen
  to be minimal in the following sense:
  if $\widetilde \Phi \in \Mult(\cH_s \otimes \widetilde \cF, \cH_k \otimes \cE)$
  is another partially isometric multiplier with $\cM = \widetilde \Phi \cdot (\cH_s \otimes \widetilde \cF)$,
  then any partial isometry $V: \cF \to \widetilde \cF$ as in part {\rm (a)} is an isometry.
  In particular, up to unitary equivalence there is a unique such minimal multiplier.
  \end{enumerate}
\end{prop}
\begin{proof}
  (a)
  If $\Phi \in \Mult(\cH_s \otimes \cF, \cH_k \otimes \cE)$ is a partially isometric multiplier
  with $\cM = \Phi \cdot (\cH_s \otimes \cE)$, then it is a co-isometric multiplier from $\cH_s \otimes \cF$
  to $\cM$, hence Equation \eqref{eqn:quotient_factored} holds for $\Phi$ by part (b)
  of Lemma \ref{lem:mult_char}. In particular, if $\Phi \in \Mult(\cH_s \otimes \cF, \cH_k \otimes \cE)$ and
  $\widetilde \Phi \in \Mult(\cH_s \otimes \widetilde \cF, \cH_k \otimes \cE)$ are two such multipliers, then
  \begin{equation*}
    \Phi(z) \Phi(w)^* = \widetilde \Phi(z) \widetilde \Phi(w)^* \quad (z,w \in X),
  \end{equation*}
  hence there exists a unitary operator
  \begin{equation*}
    V: \bigvee_{\substack{z \in X \\ \xi \in \cE}} \Phi(z)^* \xi
    \to \bigvee_{\substack{z \in X \\ \xi \in \cE}} \widetilde \Phi(z)^* \xi, \quad \Phi(z)^* \xi \mapsto \widetilde \Phi(z)^* \xi,
  \end{equation*}
  where $\bigvee$ denotes the closed linear span.
  Extending $V$ by zero on the orthogonal complements, we obtain a partial isometry which we still denote by $V$ and which satisfies $\Phi(z) = \widetilde \Phi(z) V$
  and $\widetilde \Phi(z) = \Phi(z) V^*$.

  (b) To find a minimal partially isometric multiplier $\Phi$, we note that in the
  proof of Theorem \ref{thm:main} we may clearly choose \[\cF = \bigvee_{\substack{z \in X \\ \xi \in \cE}}  \Phi(z)^* \xi.\]
  If $\widetilde \Phi$ is another multiplier and $V: \cF \to \widetilde \cF$
  is a partial isometry as in part (a), then $\Phi(z) = \Phi(z) V^* V$ for all $z \in X$ and hence $V^* V = I$ by choice of $\cF$.
\end{proof}

\subsection{Pairs of nested invariant subspaces}
Our first factorization result for pairs of nested invariant subspaces is a more general version of Theorem \ref{thm:factorization}.
Let $k$ and $\ell$ be two kernels on a set $X$ and let $\cE$ be a Hilbert space.  If $\cM \subset \cH_k \otimes \cE$ and
$\cN \subset \cH_\ell \otimes \cE$ are two closed subspaces,
we say that $\cN$ is \emph{contractively contained} in $\cM$ if $\cN \subset \cM$ as sets, and
the inclusion is a contraction.

\begin{thm}
  \label{thm:factorization_three_space}
  Let $X$ be a set, let $k,\ell$ be kernels on $X$ and let $s$ be a normalized complete Nevanlinna--Pick kernel on $X$.
  Let $\cE$ be a Hilbert space and let $\cM \subset \cH_k \otimes \cE$ and
  $\cN \subset \cH_\ell \otimes \cE$ be two non-zero closed subspaces such that $\cN$
  is contractively contained in $\cM$.
  If $\cF,\cG$ are Hilbert spaces and
  $\Phi \in \Mult(\cH_s \otimes \cF,
  \cH_k \otimes \cE)$ and $\Psi \in \Mult(\cH_s \otimes \cG, \cH_\ell \otimes \cE)$ are partially
  isometric multipliers with $\cM = \Phi \cdot (\cH_s \otimes \cF)$ and $\cN = \Psi \cdot (\cH_s \otimes \cG)$,
  then there exists a contractive multiplier $\Gamma \in \Mult(\cH_s \otimes \cG, \cH_s \otimes \cF)$
  with $\Psi = \Phi \Gamma$.
\end{thm}

\begin{proof}
  Let $k^{\cM}$ and $\ell^{\cN}$ denote the reproducing kernels of $\cM$ and $\cN$ respectively. Then part (b) of Lemma \ref{lem:mult_char}
  implies that
  \begin{equation*}
    s(z,w) \Phi(z) \Phi(w)^* = k^{\cM}(z,w)  \quad \text{ and } s(z,w) \Psi(z) \Psi(w)^* = \ell^{\cN}(z,w).
  \end{equation*}
  Since $\cN$ is contractively contained in $\cM$, we find that $\ell^{\cN} \le k^{\cM}$
  (this is the special case of part (a) of Lemma \ref{lem:mult_char} when the multiplier
  is the constant function $I_{\cE}$).
  Thus
  \begin{equation*}
    s(z,w) ( \Phi(z) \Phi(w)^* -  \Psi(z) \Psi(w)^*) \ge 0.
  \end{equation*}
  An application of Leech's theorem for complete Nevanlinna--Pick kernels (the implication (i) $\Rightarrow$ (ii)
  of \cite[Theorem 8.57]{AM02}) now yields a contractive multiplier $\Gamma$ as desired. (Strictly
    speaking, it is assumed in the statement of \cite[Theorem 8.57]{AM02} that $s$ satisfies an irreducibility
    assumption that implies that $\cH_s$ separates the points of $X$; however, the proof shows
  that it suffices to assume that $s$ is normalized.)
\end{proof}

We can now prove our generalization of Arias's solution of the McCullough and Trent conjecture.

\begin{proof}[Proof of Theorem \ref{thm:factorization2}]
  By Theorem \ref{thm:main}, there exist a Hilbert space $\cG_0$ and a partially
  isometric multiplier $\Psi \in \Mult(\cH_s \otimes \cG_0, \cH_k \otimes \cE)$
  so that $\cN = \Psi \cdot (\cH_s \otimes \cG_0)$. By Theorem \ref{thm:factorization},
  there exists a contractive multiplier $\Gamma_0 \in \Mult(\cH_s \otimes \cG_0, \cH_s \otimes \cF)$
  such that $\Psi = \Phi \Gamma_0$.
  Let 
  \[
    \cL = \ol{\Gamma_0 \cdot (\cH_s\otimes \cG_0)},
\]
    which is a closed $\Mult(\cH_s)$-invariant subspace
  of $\cH_s \otimes \cF$. Thus, we may apply Theorem \ref{thm:main} again
  to find a Hilbert space $\cG$ and a partially isometric multiplier $\Gamma \in \Mult(\cH_s \otimes \cG, \cH_s \otimes \cF)$ such that $\cL = \Gamma \cdot (\cH_s \otimes \cG)$.
  Then $\Phi \Gamma \in \Mult(\cH_s \otimes \cG, \cH_k \otimes \cE)$ is
    a contractive multiplier whose range is contained in $\cN$, so that
  \begin{equation*}
    P_{\cN} \ge M_{\Phi} M_{\Gamma} M_{\Gamma}^* M_{\Phi}^* = M_{\Phi} P_{\cL} M_{\Phi}^* 
    \ge M_{\Phi} M_{\Gamma_0} M_{\Gamma_0}^* M_{\Phi}^* = M_{\Psi} M_{\Psi}^* = P_{\cN}.
  \end{equation*}
  Consequently, $M_{\Phi \Gamma} M_{\Phi \Gamma}^* = P_{\cN}$.
\end{proof}

\subsection{Spaces of holomorphic functions in $\bC^d$}

Finally, we provide the proof that invariance of a subspace under multiplication by the coordinate
functions suffices in many concrete cases to obtain a representation as in Theorem \ref{thm:main}.  The argument is standard, see for instance \cite[Lemma 4.1]{GRS2002}. A domain $\Omega \subset \bC^d$ is called \emph{circular} if $\lambda \Omega \subset \Omega$
for all $\lambda$ in the unit circle $\bT$.

\begin{prop}
  \label{prop:coord}
  Let $\Omega\subset \bC^d$ be a circular domain containing the origin.
  Let $k,s$ be kernels on $\Omega$ such that $s$ is non-vanishing and $k/s \ge 0$. Suppose that
  \begin{itemize}
    \item $\cH_s$ consists of analytic functions on $\Omega$,
    \item the coordinate functions $z_1,\ldots,z_d$ are multipliers of $\cH_s$, and
    \item $s(\lambda z, \lambda w) = s(z,w)$ for all $z,w \in \Omega, \lambda \in \bT$.
  \end{itemize}
  Then, a closed subspace $\cM \subset \cH_k \otimes \cE$ is $\Mult(\cH_s)$-invariant
  if and only if it is invariant under multiplication by $z_1,\ldots,z_d$.
\end{prop}
\begin{proof}
  To prove the non-trivial implication, suppose that $\cM$ is invariant
  under $z_1,\ldots,z_d$ and let $\varphi \in \Mult(\cH_s)$.
  If $f: \Omega \to \bC$ and $\lambda \in \bT$, write $f_\lambda(z) = f(\lambda z)$.
  The assumptions on $s$ imply that there exists an SOT-continuous unitary representation
  \begin{equation*}
    \Gamma: \bT \to \cB(\cH_s) \quad \text{such that} \quad \Gamma_\lambda f = f_\lambda
  \end{equation*}
  (see for instance \cite[Section 6]{hartz2017isom}).
  Since $M_{\varphi_\lambda} = \Gamma_\lambda M_\varphi \Gamma_{\lambda}^*$, we find
  that 
  \[
   \|\varphi_\lambda\|_{\Mult(\cH_s)} = \|\varphi\|_{\Mult(\cH_s)}
   \]
  and that the map $\lambda \mapsto \varphi_\lambda$ is SOT-continuous.
  In this setting, a routine application of the Fej\'er kernel (cf. \cite[Lemma I 2.5]{Katznelson04} or \cite[Lemma 4.1]{GRS2002})
  shows that the Fej\'er means $(p_n)$ of $\varphi$ converge to $\varphi$ in the strong operator
  topology of $\cB(\cH_s)$. Since $\Mult(\cH_s) \subset \Mult(\cH_k)$ contractively,
  the sequence $(p_n)$ is also bounded in $\Mult(\cH_k)$,
  and since $(p_n)$ also converges to $\varphi$ pointwise, it converges to $\varphi$ at least
  in the weak operator topology of $\cB(\cH_k)$. Therefore, invariance of $\cM$ under each polynomial $p_n$
  implies invariance under $\varphi$, as asserted.
\end{proof}

\section{Discussion of further extensions}
\label{sec:further}

\subsection{Extending Theorem \ref{thm:factorization2} to triples of kernels}

Since the first factorization result for pairs of nested subspaces, Theorem \ref{thm:factorization}, readily
generalizes to triples of kernels $k,\ell,s$ (see Theorem \ref{thm:factorization_three_space}) one might ask if the second factorization result,
Theorem \ref{thm:factorization2},  can also be generalized to triples of kernels. More precisely, one could ask:

\begin{quest}
  \label{quest:factorization2_three_space}
  Let $X$ be a set, let $k,\ell$ be kernels on $X$ and let $s$ be a normalized complete Nevanlinna--Pick kernel on $X$
  such that $\ell/s \ge 0$.
  Let $\cE$ be a Hilbert space and let $\cM \subset \cH_k \otimes \cE$ and
  $\cN \subset \cH_\ell \otimes \cE$ be two non-zero closed subspaces such that $\cN$
  is contractively contained in $\cM$.

  If
  $\Phi \in \Mult(\cH_s \otimes \cF,
  \cH_k \otimes \cE)$ is a partially isometric
  multiplier with $\cM = \Phi \cdot (\cH_s \otimes \cF)$ and if $\cN$ is $\Mult(\cH_s)$-invariant,
  does there exist a Hilbert space $\cG$
  and a partially isometric multiplier $\Gamma \in \Mult(\cH_s \otimes \cG, \cH_s \otimes \cF)$
  such that $\Phi \Gamma$ is a partially isometric multiplier from $\cH_s \otimes \cG$
  into $\cH_\ell \otimes \cE$ with $\cN = (\Phi \Gamma) \cdot (\cH_s \otimes \cG)$?
\end{quest}

The following somewhat trivial example shows that this is not possible in general.

\begin{exa}
  Let $X = \bD$ be the open unit disc in $\bC$ and let $s$ be the constant function $1$, which we view as a kernel on $X$. Then, $\cH_s$ simply consists of all constant functions.
  Let $\cH_\ell = H^2$ (the Hardy space on $\bD$), and let $\cH_k = A^2$ (the Bergman space on $\bD$).
  Let further $\cE = \bC$, $\cM = A^2$ and $\cN = H^2$, so that $\cN$ is indeed contractively
  contained in $\cM$ as $\ell \le k$.

  Let $\cF = A^2$ and define $\Phi:\bD\to \cB(\cF,\bC)$ by $\Phi(z) = k_z^* \in \cB(\cF,\bC)$.
  (Note that in this example, $A^2$ plays both the role of a reproducing kernel Hilbert space on $\bD$ and 
  that of a coefficient space.) Since
  \begin{equation*}
    k(z,w) = \Phi(z) s(z,w) \Phi(w)^* \quad (z, w \in \bD),
  \end{equation*}
  Lemma \ref{lem:mult_char} (b) shows that $\Phi$ is a co-isometric multiplier from $\cH_s \otimes \cF$
  onto $\cH_k$. In fact, since $\cH_s$ consists of all constant functions, $\cH_s \otimes \cF$
  is canonically identified with $\cF = A^2 = \cH_k$, and $M_{\Phi}$ is simply the identity operator
  modulo this identification, so in particular $M_{\Phi}$ is a unitary from $\cH_s \otimes \cF$ onto $\cH_k$.

  We claim that there do not exist a Hilbert space $\cG$ and a partially isometric multiplier $\Gamma$ from 
  $\cH_s \otimes \cG$ into $\cH_s \otimes \cF$ such that $\Phi \Gamma$ is a multiplier
  from $\cH_s \otimes \cG$ into $\cH_\ell$ with $\cN = (\Phi \Gamma) \cdot (\cH_s \otimes \cG)$.
  Indeed, if $\Gamma$ is a partially isometric multiplier from $\cH_s \otimes \cG$ into $\cH_s \otimes \cF$,
  then since $M_{\Phi}: \cH_s \otimes \cF \to \cH_k$
  is unitary, the range of $M_{\Phi} M_{\Gamma}$ is
  a closed subspace of $\cH_k = A^2$, and in particular not equal to $\cN = H^2$.
\end{exa}

\subsection{Boundary values}

Let $\cE$ and $\cF$ be separable Hilbert spaces.
It is an important part of the classical Beurling--Lax--Halmos theorem that every isometric multiplier $\Phi \in \Mult(H^2 \otimes \cF, H^2 \otimes \cE)$
is inner in the sense that
its non-tangential boundary values $\Phi(z)$ are isometries for almost every $z \in \bT$.
Greene, Richter and Sundberg \cite{GRS2002} strengthened the analogy between the theorem of McCullough and Trent \cite{MT00}
and the classical Beurling--Lax--Halmos theorem.
They showed that for a large class of complete Nevanlinna--Pick spaces $\cH$
on the unit ball $\bB_d$ in $\bC^d$, every partially isometric multiplier $\Phi \in \Mult(\cH \otimes \cF, \cH \otimes \cE)$
has non-tangential boundary values that are partial isometries of constant rank
\begin{equation}
  \label{eqn:rank}
  r = \sup \{ \operatorname{rank} \Phi(z): z \in \bB_d \}
\end{equation}
almost everywhere on $\partial \bB_d$.

One may ask if there is a result of this type in our context. Assume henceforth the setting of Theorem \ref{thm:main}. Suppose that $X = \bB_d$ and also assume for simplicity that
$k(z,z) \neq 0$ for all $z \in X$. A multiplier $\Phi \in \Mult(\cH_s \otimes \cF, \cH_k \otimes \cE)$ is not necessarily bounded; instead,
if it is a contractive mutliplier, then by virtue of Lemma \ref{lem:mult_char} it obeys the estimate
\begin{equation*}
  \|\Phi(z) \|^2_{\cB(\cF,\cE)} \le \frac{k(z,z)}{s(z,z)} \quad (z \in X).
\end{equation*}
Therefore, if $\Phi$ is a contractive multiplier, then the range of the function
\begin{equation*}
  G_{\Phi}: \bB_d \to \cB(\cE), \quad z \mapsto \frac{s(z,z)}{k(z,z)} \Phi(z) \Phi(z)^*
\end{equation*}
consists of positive contractions on $\cE$. The result of Greene, Richter, Sundberg \cite{GRS2002}
then shows that if $k=s$ belongs to their class of complete Nevanlinna--Pick kernels on $\bB_d$, then
for every partially isometric multiplier $\Phi$, the function $G_{\Phi}$ has non-tangential boundary
values that are orthogonal projections of constant rank $r$ as in \eqref{eqn:rank}.

We now indicate why such a result fails when $\cH_s = H^2$, $\cH_k = A^2$,
and $\cE = \bC$.
Let $\cM$ be a non-zero multiplier invariant subspace of $A^2$ and let $\Phi \in \Mult(H^2 \otimes \cF, A^2)$ be a
partially isometric multiplier with $\cM = \Phi \cdot (H^2 \otimes \cF)$.
In this setting, the direct analogue of the theorem of Greene, Richter and Sundberg would be the statement
that the scalar-valued function $G_{\Phi}$ has non-tangential limit $1$ almost everywhere on $\bT$.

Part (b) of Lemma \ref{lem:mult_char}
implies that
\begin{equation*}
  \Phi(z) \Phi(w)^* = \frac{k^{\cM}(z,w)}{s(z,w)} \quad (z,w \in \bD),
\end{equation*}
and hence
\begin{equation*}
  G_{\Phi}(z) = \frac{k^{\cM}(z,z)}{k(z,z)} \quad(z \in \bD).
\end{equation*}
This function is the square of the majorization function of \cite{ARS02} and is called the root function in \cite{YZ03}.
In some very simple cases of invariant subspaces of $A^2$, such as the invariant subspaces generated by $(z-a)^N$ for $a \in \bD$
and $N \in \bN$, the function $G_{\Phi}$ can be computed explicitly and indeed has boundary values $1$ almost everywhere
on $\bT$, see \cite[Proposition 12]{YZ03}.

For general multiplier invariant subspaces $\cM$, however, the question of boundary values of $G_\Phi$ is a delicate one, as can
be seen from \cite{ARS02}. In particular, it follows from Theorem A of \cite{ARS02} that if $\cM$ is contained
in an invariant subspace of index larger than one, then $G_{\Phi}$ has a non-tangential limit inferior of $0$
almost everywhere on $\bT$, see \cite{ARS02} for definitions and further discussions. Furthermore,
Proposition 7.1 of \cite{ARS02} (and the discussion preceding it) shows that if $\Lambda$ is an $A^2$-interpolating sequence with non-tangential cluster
set $E \subset \bT$, and if
\begin{equation*}
  \cM = \{ f \in A^2: f(z) = 0 \text{ for all } z \in \Lambda \},
\end{equation*}
then $G_{\Phi}$ has non-tangential limit $1$ for almost every $z \in \bT \setminus E$, but
has non-tangential limit inferior $0$ for all $z \in E$. Moreover, the proof of Corollary 7.4 in \cite{ARS02}
shows that for any compact subset $E \subset \bT$, there is an $A^2$-interpolating sequence whose non-tangential cluster set is $E$.
Thus, the boundary behavior of $G_\Phi$ can be very complicated.
 
\bibliographystyle{amsplain}
\bibliography{literature}

\end{document}